\documentclass[journal]{IEEEtran}
\usepackage{latexsym, amssymb,amsmath, amsbsy, amsopn, amstext, amsthm}
\usepackage{graphicx}
\usepackage{ifpdf}
\usepackage[center]{caption}
\usepackage{tikz}
\usepackage[\ifpdf pdftex \else dvipdfm \fi]{hyperref}
\usepackage[numbers,sort&compress]{natbib}
\hypersetup{CJKbookmarks,%
       bookmarksnumbered,%
            pdfstartview=FitH,
           bookmarksopen=true,
         linkbordercolor={1 1 1},
              colorlinks=false,
               filecolor=blue,linkcolor=blue,citecolor=blue,anchorcolor=green,urlcolor=blue}

\DeclareMathOperator{\Col}{Col}
\DeclareMathOperator{\lcm}{lcm}

\def\ra{\rightarrow}

\def\lra{\leftrightarrow}

\def\a{\alpha}
\def\b{\beta}
\def\d{\delta}

\def\D{\Delta}

\newcommand{\R}{{\mathbb R}}

\def\dsum{\mathop{\sum}\limits}

\graphicspath{{./images/}}

\newtheorem{thm}{Theorem}[section]

\newtheorem{dfn}[thm]{Definition}
\newtheorem{prp}[thm]{Proposition}
\newtheorem{exa}[thm]{Example}
\newtheorem{rem}[thm]{Remark}

\begin{document}

\title{Observability of Boolean Networks via Set Controllability Approach}

\author{Daizhan Cheng, Changxi Li, Fenghua He
\thanks{This work is supported in part by the National Natural Science Foundation (NNSF) of
China under Grants 61333001, 61773371, 61733018.}
\thanks{D. Cheng is with the Key Laboratory of Systems and Control, Academy of
Mathematics and Systems Sciences, Chinese Academy of Sciences, Beijing 100190, P.R.China (E-mail:
dcheng@iss.ac.cn). Changxi Li and Fenghua He are with Institute of Astronautics, Harbin Institute of Technology, Harbin, P.R.China}
}

\maketitle

\begin{abstract}%
The controllability and observability of Boolean control network(BCN) are two fundamental properties. But the verification of latter is much harder than the former. This paper considers the observability of BCN via controllability. First, the set controllability is proposed, and  the necessary and sufficient condition is obtained. Then a technique is developed to convert the observability into an equivalent set controllability problem. Using the result for set controllability, the necessary and sufficient condition is also obtained for the observability of BCN.
\end{abstract}

\begin{IEEEkeywords}
Boolean control network, set controllability, observability, semi-tensor product of matrices.
\end{IEEEkeywords}

\section{Introduction}

\IEEEPARstart{B}{oolean} network was firstly proposed by Kauffman to describe gene regularity networks \cite{kau69}. Since then it has attracted much attention from biologists, physicists, and system scientists \cite{aku07,alb00,fau06}.

Recently, a new matrix product, called the semi-tensor product (STP) of matrices was introduced. STP has then been successfully applied to modeling and controlling Boolean networks \cite{che11,che12,lujpr}. Inspired by STP, the theory of Boolean control networks (BCNs) as well as the control of general logical systems have been developed rapidly. A set of systematic results have been obtained. For instance, the controllability and observability of Boolean networks have been discussed in \cite{che09,las12}; the disturbance decoupling has been considered in \cite{che11b,yan13}; the optimal control has been investigated in \cite{for14,las11}; the stability and stabilization have been studied in \cite{che11c,liu16}, just to mention a few.

Among them the controllability and observability of BCN are of particular importance. Particularly, the controllability via free control sequence is fundamental, and it has been solved elegantly by \cite{zha10}. Unlike the controllability, the observability has also been discussed for long time and various kinds of observability have been proposed and investigated \cite{che09,zha10,for13,las13}. A comparison for various kinds of observability has been presented in \cite{zha16}. Moreover, \cite{zha16} has also pointed out that one of them, which will be specified later, is the most sensitive observability. Here ``most sensitive one" means all other kinds of observability implies this one. In addition, \cite{zha16} has also provided  necessary and sufficient conditions for various kinds of observability via finite automata approach. Motivated by the idea of \cite{zha16}, \cite{che16} proposed a numerical method to verify the (most sensitive) observability.

Since for BCN the controllability is much easier understandable and verifiable than the observability, this paper proposes a method to verify observability via controllability. In this paper the set controllability of BCN is proposed first. The idea comes from \cite{las12}, where some states are forbidden and it is a special case of our set controllability.  Hence the result about  set controllability can be considered as a generalization of the corresponding result in \cite{las12}. Then a properly designed extended system of the original BCN is built. It is proved that the observability of the original BCN is equivalent to the set controllability of the extended system. Then the observability of BCN is converted into a set controllability problem and  then is solved completely. In fact, the result is equivalent to the necessary and sufficient condition proposed in \cite{che16}. But the new result is concise and easily verifiable.

The rest of this paper is organized as follows: Section 2 describes the set controllability of BCN. The set controllability matrix is constructed. Using it an easily verifiable necessary and sufficient condition is obtained. As an application, the output controllability problem is also solved.  Section 3 constructs an extended system and carefully designs the initial and destination sets. Then the observability of a BCN becomes the set controllability of its extended system. Some examples are presented to describe the design procedure. Section 4 is a concluding remark.

Before ending this section, a list of notations is presented:

\begin{enumerate}

\item[(1)]  ${\cal M}_{m\times n}$: the set of $m\times n$ real matrices.

\item[(2)] $\Col(M)$: the set of columns  of a matrix $M$. $\Col_i(M)$ : the $i$-th column  of $M$.

\item[(3)] ${\cal D}:=\{0,1\}$.

\item[(4)] $\d_n^i$: the $i$-th column of the identity matrix $I_n$.

\item[(5)] $\D_n:=\left\{\d_n^i\vert i=1,\cdots,n\right\}$.

\item[(6)] ${\bf 1}_{\ell}=(\underbrace{1,1,\cdots,1}_{\ell})^T$.


\item[(7)] A matrix $L\in {\cal M}_{m\times n}$ is called a logical matrix
if the columns of $L$ are of the form $\d_m^k$. That is, $\Col(L)\subset \D_m$.
Denote by ${\cal L}_{m\times n}$ the set of $m\times n$ logical matrixes.

\item[(8)] If $L\in {\cal L}_{n\times r}$, by definition it can be expressed as
$L=[\d_n^{i_1},\d_n^{i_2},\cdots,\d_n^{i_r}]$. For the sake of
compactness, it is briefly denoted as $
L=\d_n[i_1,i_2,\cdots,i_r]$.

\item[(9)] Denote by ${\cal B}_{m\times n}$ the set of $m\times n$ Boolean matrices.

\item[(10)] Let $A, B\in {\cal B}_{m\times n}$. Then $A+_{{\cal B}}B$ is the Boolean addition  (with respect to $+_{{\cal B}}=\vee$ and $\times_{{\cal B}}=\wedge$).

\item[(11)] Let $A\in {\cal B}_{m\times n}$, $B\in {\cal B}_{p\times q}$. Then $A\ltimes_{{\cal B}}B$ is the Boolean (semi-tensor) product (with respect to $+_{{\cal B}}=\vee$ and $\times_{{\cal B}}=\wedge$).

\item[(12)] $A^{(k)}:=\underbrace{A\ltimes_{{\cal B}}\cdots\ltimes_{{\cal B}}A}_{k}$.


\item[(13)] $P^0\subset 2^N$ is the set of initial sets, where $N=\{1,2,\cdots,n\}$ is the set of state nodes of a BCN, and $2^N$ is the power set of  $N$.

\item[(14)] $P^d\subset 2^N$ is the set of destination sets.

\end{enumerate}

\section{Set Controllability}

A Markov-type Boolean control network with $n$ nodes is described as \cite{che11}
\begin{align}\label{2.1}
\begin{array}{l}
\begin{cases}
x_1(t+1)=f_1(x_1(t),\cdots,x_n(t);u_1(t),\cdots,u_m(t))\\
x_2(t+1)=f_2(x_1(t),\cdots,x_n(t);u_1(t),\cdots,u_m(t))\\
\vdots\\
x_n(t+1)=f_n(x_1(t),\cdots,x_n(t);u_1(t),\cdots,u_m(t)),
\end{cases}\\
~~~y_j(t)=h_j(x_1(t),\cdots,x_n(t)),\quad j=1,\cdots, p,
\end{array}
\end{align}
where $x_i\in {\cal D}$, $i=1,\cdots,n$ are state variables; $u_i\in {\cal D}$, $i=1,\cdots,m$ are controls; $y_j\in {\cal D}$, $j=1,\cdots,p$ are outputs; $f_i:{\cal D}^{m+n}\ra {\cal D}$, $i=1,\cdots,n$, and $h_j:{\cal D}^{n}\ra {\cal D}$, $j=1,\cdots,p$ are Boolean functions.

\begin{dfn}\label{d2.1} The system (\ref{2.1}) is
\begin{enumerate}
\item controllable from $x_0$ to $x_d$, if there are a $T>0$ and a sequence of control $u(0),\cdots,u(T-1)$, such that driven by these controls  the trajectory can go from $x(0)=x_0$ to $x(T)=x_d$;
\item controllable at $x_0$, if it is controllable from $x_0$ to destination $x_d=x$, $\forall x$;
\item controllable, if it is controllable at any $x_0$.
\end{enumerate}
\end{dfn}

Under the vector form expression:
$$
1\sim \d_2^1, 0\sim \d_2^2,
$$ we have  $x_i,~u_i,~y_j\in \D_2$. Using Theorem \ref{ta.8}, (\ref{2.1}) can be converted into its algebraic form as
\begin{align}\label{2.2}
\begin{array}{l}
\begin{cases}
x(t+1)=Lu(t)x(t)\\
y(t)=Hx(t),
\end{cases}
\end{array}
\end{align}
where $x(t)=\ltimes_{i=1}^nx_i(t)$, $y(t)=\ltimes_{i=1}^py_i(t)$,  $u(t)=\ltimes_{j=1}^mu_j(t)$, and $L\in {\cal L}_{2^n\times 2^{n+m}}$, $H\in {\cal L}_{2^p\times 2^{n}}$.

Define
\begin{align}\label{2.3}
M:={\dsum_{\cal B}}_{j=1}^{2^m}L\d_{2^m}^j,
\end{align}
and set
\begin{align}\label{2.4}
{\cal C}:={\dsum_{\cal B}}_{i=1}^{2^n}M^{(i)},
\end{align}
which is called the controllability matrix. Then we have the following result:

\begin{thm}\label{t2.2}\cite{zha10} Consider the  controllability of system (\ref{2.1}) (by free control sequence). Assume its controllability matrix is ${\cal C}=(c_{i,j})$, then we have the following results:

\begin{enumerate}
\item State $x_i$ is controllable from $x_j$, if and only if, $c_{i,j}=1$.
\item System (\ref{2.1}) is controllable at $x_j$, if and only if, $\Col_j({\cal C})={\bf 1}_{2^n}$.
\item System (\ref{2.1}) is controllable, if and only if, ${\cal C}={\bf 1}_{2^n\times 2^n}$.
\end{enumerate}
\end{thm}

Denote by $N=\{1,2,\cdots,n\}$ the set of state nodes. Assume $s\in 2^N$, the index vector of $s$, denoted by $V(s)\in \R^n$, is defined as
$$
\left(V(s)\right)_i=\begin{cases}
1,\quad i\in s\\
0,\quad i\not\in s.
\end{cases}
$$

Define the set of initial sets $P^0$ and the set of destination sets $P^d$ respectively as follows:
\begin{align}\label{2.5}
\begin{array}{ccl}
P^0&:=&\left\{s^0_1,s^0_2,\cdots,s^0_{\a}\right\}\subset 2^N,\\
P^d&:=&\left\{s^d_1,s^d_2,\cdots,s^d_{\b}\right\}\subset 2^N.
\end{array}
\end{align}

Using initial sets and destination sets, the set controllability is defined as follows.

\begin{dfn}\label{d2.3} Consider system (\ref{2.1})  with a set of initial sets $P^0$ and a set of destination sets $P^d$.  The system (\ref{2.1}) is
\begin{enumerate}
\item set controllable from $s^0_j\in P^0$ to $s^d_i\in P^d$, if there exist $x_0\in s^0_j$ and $x_d\in s^d_i$, such that $x_d$ is controllable from $x_0$;
\item set controllable at $s^0_j$, if for any $s^d_i\in P^d$, the system is controllable from $s^0_j$ to $s^d_i$;
\item set controllable, if it is set controllable at any $s^0_j\in P^0$.
\end{enumerate}
\end{dfn}

Using the set of initial sets and the set of destination sets defined in (\ref{2.5}), we can define the initial index matrix $J_0$ and the destination index matrix $J_d$ respectively as
\begin{align}\label{2.6}
\begin{array}{ccl}
J_0&:=&\begin{bmatrix}
V(s^0_1)&V(s^0_2)&\cdots&V(s^0_{\a})
\end{bmatrix}\in {\cal B}_{2^n\times \a};\\
J_d&:=&\begin{bmatrix}
V(s^d_1)&V(s^d_2)&\cdots&V(s^d_{\b})
\end{bmatrix}\in {\cal B}_{2^n\times \b}.
\end{array}
\end{align}

Using (\ref{2.6}), we define a matrix, called the set controllability matrix, as
\begin{align}\label{2.7}
{\cal C}_S:=J_d^T\times_{{\cal B}}{\cal C}\times_{{\cal B}}J_0 \in {\cal B}_{\b\times \a}.
\end{align}

Note that hereafter all the matrix products are assumed to be Boolean product ($\times_{{\cal B}}$). Hence the symbol $\times_{{\cal B}}$ is omitted.

According to the definition of set controllability, the following result is easily verifiable.

\begin{thm}\label{t2.4} Consider system (\ref{2.1}) with the set of initial sets $P^0$ and the set of destination sets $P^d$ as defined in (\ref{2.5}). Moreover, the corresponding set controllability matrix  ${\cal C}_S=(c_{ij})$ is defined in (\ref{2.7}). Then
\begin{enumerate}
\item system (\ref{2.1}) is set controllable from $s^0_j$ to $s^d_i$, if and only if, $c_{i,j}=1$;
\item system (\ref{2.1}) is set controllable at $s^0_j$, if and only if $\Col_j\left({\cal C}_S\right)={\bf 1}_{\b}$;
\item system (\ref{2.1}) is set controllable, if and only if, ${\cal C}_S={\bf 1}_{\b\times \a}.$
\end{enumerate}
\end{thm}

\begin{exa}\label{e2.5} Consider the following system \cite{che11}
\begin{align}\label{2.8}
\begin{array}{l}
\begin{cases}
x_1(t+1)=(x_1(t)\lra x_2(t))\vee u_1(t)\\
x_2(t+1)=\neg x_1(t)\wedge u_2(t),
\end{cases}\\
~~~y(t)=x_1(t)\wedge x_2(t).
\end{array}
\end{align}
It is easy to calculate that
$$
{\cal C}=\begin{bmatrix}
1&1&1&1\\
1&1&1&1\\
0&0&1&0\\
1&1&1&1\\
\end{bmatrix}.
$$
\begin{enumerate}
\item Assume
\begin{align}\label{2.801}
\begin{cases}
P^d=\left\{ s^d_1=\{\d_4^1,\d_4^2\},~~s^d_2=\{\d_4^3,\d_4^4\}\right\};\\
P^0=\left\{s^0_1=\{\d_4^1\},~~s^0_2=\{\d_4^2,\d_4^3,\d_4^4\}\right\}.\\
\end{cases}
\end{align}
Then we have
$$
J_d=\begin{bmatrix}
1&0\\
1&0\\
0&1\\
0&1
\end{bmatrix},\quad
J_0=\begin{bmatrix}
1&0\\
0&1\\
0&1\\
0&1
\end{bmatrix}.
$$
It follows that
$$
{\cal C}_S=J_d^T{\cal C}J_0=\begin{bmatrix}
1&1\\1&1\end{bmatrix}
$$
Hence, the system (\ref{2.8}) is set controllable with respect to the initial set $P^0$ and the destination set $P^d$ defined by (\ref{2.801}).

\item Assume
\begin{align}\label{2.802}
\begin{cases}
P^d=\left\{s^d_1=\{\d_4^3\}\right\};\\
P^0=\left\{s^0_1=\{\d_4^1,\d_4^2,\d_4^3\}, s^0_2=\{\d_4^1,\d_4^4\}\right\}.
\end{cases}
\end{align}
Then
$$
J_d=\begin{bmatrix}
0\\
0\\
1\\
0
\end{bmatrix},\quad
J_0=\begin{bmatrix}
1&1\\
1&0\\
1&0\\
0&1
\end{bmatrix}.
$$
And
$$
{\cal C}_S=J_d^T{\cal C}J_0=\begin{bmatrix}
1&0\end{bmatrix}.
$$
Hence, the system (\ref{2.8}) is not set controllable with respect to the initial set $P^0$ and destination set $P^d$ defined by (\ref{2.802}).
\end{enumerate}
\end{exa}

As an application, we consider the output controllability \cite{oga97}.

\begin{dfn}\label{d3.1} Consider system (\ref{2.1}). It is said to be output controllable, if for any $x(0)=x_0$ and any $y_d$, there exist a $T>0$ and a sequence of control $u(0),u(1),\cdots,u(T-1)$ such that $y(T)=y_d$.
\end{dfn}

\begin{dfn}\label{d3.2} Consider system (\ref{2.1}).
\begin{enumerate}
\item
A partition is called an output-based partition, if
\begin{align}\label{3.1}
s^d_j=\left\{x\;\big|\;Hx=\d_{2^p}^j\right\},\quad j=1,\cdots,2^p.
\end{align}
\item
A partition is called a finest partition, if
\begin{align}\label{3.2}
s^0_i=\{x_i\},\quad i=1,\cdots,2^n.
\end{align}
\end{enumerate}
\end{dfn}

Using (\ref{3.1}) and (\ref{3.2}), we define
\begin{align}\label{3.3}
\begin{cases}
P^d:=\left\{s^d_j\;\big|\; j=1,\cdots,2^p\right\};\\
P^0:=\left\{s^0_i\;\big|\; i=1,\cdots,2^n\right\}.
\end{cases}
\end{align}

Taking the construction of $P^d$ and $P^0$ into consideration, the following result is an immediate consequence of the definition.

\begin{thm}\label{t3.3} System (\ref{2.1}) is output controllability, if and only if, it is set controllability with respect to the set pairs $(P^d,~P^0)$, defined in (\ref{3.3}).
\end{thm}

Note that corresponding to $P^0$, defined in (\ref{3.3}), the initial index matrix is an identity matrix, and the destination index matrix is $H^T$. Hence for output controllability, denoting by ${\cal C}_Y$ the output controllability matrix,  we have
\begin{align}\label{3.4}
{\cal C}_Y={\cal C}_S=H{\cal C},
\end{align}
where ${\cal C}_S$ is the set controllability matrix with respect to the set pair $(P^d,~P^0)$ defined in (\ref{3.3}).

The output controllability has been discussed in \cite{liu14}. Comparing our result with the direct approach in \cite{liu14}, the advantage of set controllability approach is obvious.

\begin{exa}\label{e3.4}
Consider system (\ref{2.8}) again. It is easy to figure out that
$$
J_d=\left(\d_2[1,2,2,2]\right)^T.
$$
Then
$$
{\cal C}_Y=J_d^T{\cal C}=\begin{bmatrix}
1&1&1&1\\
1&1&1&1\\
\end{bmatrix}>0.
$$
Hence, system (\ref{2.8}) is output controllable.
\end{exa}


\section{Observability via Set Controllability Approach}

As discussed in \cite{zha16}  the following one is the most sensitive observability among those in recent literature.

\begin{dfn} \label{d5.1} \cite{zha16} System (\ref{2.1}) is observable, if for any two initial states $x_0\neq z_0$, there exist an integer $T\geq 0$ and a control sequence $u=\{u(0),u(1),\cdots,u(T-1)\}$, such that the corresponding output sequence $y(i)=y_i(x_0,u)$, $i=0,1,\cdots,T$ is not equal to $\tilde{y}_i(z_0,u)$.
\end{dfn}

Next, we consider two kinds of state pairs.

\begin{dfn}\label{d5.2} A pair $(x,z)\in \D_{2^n}\times \D_{2^n}$ is $y$-indistinguishable if $Hx=Hy$. Otherwise, $(x,z)$ is called $y$-distinguishable.
\end{dfn}

Following \cite{che16}, we split the product state space $\D_{2^n}\times \D_{2^n}$ into a partition of three components as
\begin{align}\label{5.1}
D=\{zx\;\big|\;z = x\},
\end{align}
\begin{align}\label{5.2}
\Theta=\{zx\;\big|\;z\neq x~\mbox{and}~Hz=Hx\},
\end{align}
\begin{align}\label{5.3}
\Xi=\{zx\;\big|\;Hz\neq Hx\}.
\end{align}

Using algebraic form (\ref{2.2}), we construct a dual system as
\begin{align}\label{5.4}
\begin{cases}
z(t+1)=Lu(t)z(t)\\
x(t+1)=Lu(t)x(t).
\end{cases}
\end{align}

 Then the observability problem  of system (\ref{2.1}) can be converted into a set controllability problem of the extended system (\ref{5.4}). Construct the initial sets and the destination sets as follows:

\begin{align}\label{5.5}
P^0:=\bigcup_{zx\in \Theta}\{zx\}
\end{align}
and
\begin{align}\label{5.6}
P^d:=\{\Xi\}.
\end{align}

Note that (\ref{5.5}) means that each $zx\in \Theta$ is an element of $P^0$, while (\ref{5.6}) means $P^d$ has only one element, which is $\Xi$.

Then we have the following result:

\begin{thm}\label{t5.2} System (\ref{2.1}) is observable, if and only if, system (\ref{5.4}) is set controllable from $P^0$ to $P^d$, which are defined in (\ref{5.5}) and (\ref{5.6}) respectively.
\end{thm}

\begin{proof}
(Necessary): Assume the system is observable. Then for any two initial points $z_0\neq x_0$, there exists a control sequence $\{u(t)\;|\;t=0,1,\cdots\}$ such that the corresponding output sequences $\{y(t)\;|\;t=0,1,\cdots\}$ and $\{\tilde{y}(t)\;|\;t=0,1,\cdots\}$ are not the same. Let $T\geq 0$ be the smallest $t$ such that $y(t)\neq \tilde{y}(t)$. If $T=0$, $(z_0,~x_0)\in \Xi^c$ is a distinguishable pair. Assume $T>0$. Applying the sequence of controls to system (\ref{5.4}), $(z_0,~x_0)$ can be driven to $(z(T),~x(T))$. Since $Hz(T)=y(T)\neq \tilde{y}(T)=Hx(T)$, we have $(z(T),~x(T))\in \Xi$. That is, system (\ref{5.4}) is set controllable from $P^0$ to $P^d$.

(Sufficiency):  Assume a pair $z_0\neq x_0$ is given. If $(z_0,~x_0)\in \Xi$, we are done. Otherwise, since the system (\ref{5.4}) is set controllable from $P^0$ to $P^d=\{\Xi\}$, there exists a control sequence $\{u(t)\;|\;t=0,1,\cdots\}$ which drives  $(z_0,~ x_0)$ to $(z_T,~x_T)\in \Xi$.

It is worth noting that system (\ref{5.4}) is essentially a combination of two independent systems corresponding to $z$ and $x$ respectively. Only the same control sequence is applied to them. Hence we have $z_T$ is on the trajectory of (\ref{2.2}) with the above mentioned control sequence $\{u(t)\;|\;t=0,1,\cdots\}$, that is, $z_T=x(z_0,u(0),u(1),\cdots,u(T-1))$, and $x_T=x(x_0,u(0),u(1),\cdots,u(T-1))$. Since $(z_T,~x_T)\in \Xi$, which means that using this control sequence to system (\ref{2.2}), it distinguishes $z_0$ and $x_0$.

\end{proof}
\begin{exa}\label{e5.4}
Consider the reduced model for the lac operon in the bacterium Escherichia coli \cite{avc11}
\begin{align}\label{5.7}
\begin{array}{l}
\begin{cases}
x_1(t+1)=\neg u_1(t)\wedge\left(x_2(t)\vee x_3(t)\right)\\
x_2(t+1)=\neg u_1(t)\wedge u_2(t)\wedge x_1(t)\\
x_3(t+1)=\neg u_1(t)\wedge \left(u_2(t)\vee (u_3(t)\wedge x_1(t))\right),\\
\end{cases}
\end{array}
\end{align}
where  $x_1, x_2$ and $x_3$ represent  the lac {\rm mRNA}, the lactose in high and medium concentrations, respectively;
$u_1, u_2$ and $u_3$ are  the extracellular glucose, high and medium extracellular lactose, respectively.

\begin{enumerate}
  \item
Assume that the outputs are
\begin{align}\label{5.8}
\begin{array}{l}
\begin{cases}
y_1(t)=x_1(t)\vee \neg x_2(t)\vee x_3(t)\\
y_2(t)=\neg x_1(t)\vee x_2(t)\wedge\neg x_3(t)\\
y_3(t)=\neg x_1(t)\wedge \neg x_2(t)\vee x_3(t).\\
\end{cases}
\end{array}
\end{align}

Its algebraic form is
\begin{align}\label{5.8}
\begin{array}{l}
x(t+1)=Lu(t)x(t)\\
y(t)=Hx(t),
\end{array}
\end{align}
where
$$
\begin{array}{ccl}
L&=&\d_8[8,8,8,8,8,8,8,8,8,8,8,8,8,8,8,8,\\
~&~&~~~8,8,8,8,8,8,8,8,8,8,8,8,8,8,8,8,\\
~&~&~~~1,1,1,5,3,3,3,7,1,1,1,5,3,3,3,7,\\
~&~&~~~3,3,3,7,4,4,4,8,4,4,4,8,4,4,4,8],\\
H&=&\d_8[8,6,3,6,5,6,7,6].
\end{array}
$$
Construct the dual system as
\begin{align}\label{5.9}
\begin{cases}
z(t+1)=Lu(t)z(t)\\
x(t+1)=Lu(t)x(t).
\end{cases}
\end{align}
It is easy to figure out that
$$
\begin{array}{cl}
\Theta=&\left\{\{\d_8^2,\d_8^4\}, \{\d_8^2,\d_8^6\},\{\d_8^2,\d_8^8\},\{\d_8^4,\d_8^6\},\right.\\
~&~\left.\{\d_8^4,\d_8^8\},\{\d_8^6,\d_8^8\}\right\}\\
~\sim&\left\{\d_{64}^{12}, \d_{64}^{14}, \d_{64}^{16},\d_{64}^{30},\d_{64}^{32},\d_{64}^{48}\right\}\\
~:=&\{\theta_1,\theta_2, \theta_3,\theta_4,\theta_5, \theta_6\};
\end{array}
$$
and
$$
\begin{array}{cl}
\Xi=&\left\{\{\d_8^1,\d_8^2\},\{\d_8^1,\d_8^3\},\{\d_8^1,\d_8^4\}, \{\d_8^1,\d_8^5\},\{\d_8^1,\d_8^6\},\right.\\
~&~\left.\{\d_8^1,\d_8^7\},\{\d_8^1,\d_8^8\},\{\d_8^2,\d_8^3\},\{\d_8^2,\d_8^5\}, \{\d_8^2,\d_8^7\},\right.\\
~&~\left.\{\d_8^4,\d_8^3\},\{\d_8^4,\d_8^5\},\{\d_8^4,\d_8^7\},\{\d_8^6,\d_8^3\},\{\d_8^6,\d_8^5\}, \right.\\
~&~\left.\{\d_8^6,\d_8^7\},\{\d_8^8,\d_8^3\},\{\d_8^8,\d_8^5\},\{\d_8^8,\d_8^7\},\{\d_8^3,\d_8^5\},\right.\\
~&~\left.\{\d_8^3,\d_8^7\},\{\d_8^5,\d_8^7\}\right\}\\
~~\sim&\left\{\d_{64}^{2}, \d_{64}^{3}, \d_{64}^{4},\d_{64}^{5},\d_{64}^{6},\d_{64}^{7},\d_{64}^{8}, \d_{64}^{11}, \d_{64}^{13},\d_{64}^{15},\d_{64}^{21},\right.\\
~&~\left.\d_{64}^{23},\d_{64}^{27},\d_{64}^{29},\d_{64}^{31},\d_{64}^{39},\d_{64}^{43},\d_{64}^{45},\d_{64}^{47},\d_{64}^{59},\d_{64}^{61},\d_{64}^{63}\right\}.
\end{array}
$$

Set $w(t)=z(t)x(t)$, then (\ref{5.9}) can be expressed as
$$
\begin{array}{l}
z(t+1)=L\left(I_{64}\otimes {\bf 1}_{8}^T\right)u(t)w(t)\\
x(t+1)=L\left({\bf 1}_{8}^T\otimes I_{8}\right)u(t)w(t).
\end{array}
$$
Finally, we have
\begin{align}\label{5.10}
w(t+1)=Mu(t)w(t),
\end{align}
where
$$
\begin{array}{ccl}
M&=&\d_{64}[ 64, 64, 64,  \ldots,  60, 60, 60, 64]\in{\cal L}_{{64}\times {512}}\\
~&:=&[M_1,M_2,M_3,M_4,M_5,M_6,M_7,M_8].
\end{array}
$$
Then the controllability matrix of (\ref{5.9}) can be calculated by
$$
{\cal C}:={\dsum_{{\cal B}}}_{j=1}^{64}\left({\dsum_{{\cal B}}}_{i=1}^8M_i\right)^{(j)}\in {\cal B}_{64\times 64}.
$$

Finally, we consider the set controllability of (\ref{5.9}).  Using the initial set $P^0=\{\theta\in\Theta\}=\{\theta_1,\theta_2,\theta_3,\theta_4,\theta_5,\theta_6\}$ and the destination set $P^d=\Xi$, we have
$$
J_d=\sum_{\d_{64}^i\in \Xi}\d_{64}^i;
$$
and
$$
J_0= \d_{64}[ 12,  14,  16,  30,  32,  48].
$$
It follows that
$$
{\cal C}_S=J_d^T{\cal C}J_0=\begin{bmatrix}1&1&1&1&1&1\end{bmatrix}>0.
$$
According to Theorem \ref{t5.2}, system (\ref{5.7}) with outputs (\ref{5.8}) is observable.

\item Assume the measured outputs of system (\ref{5.7}) are
\begin{align}\label{5.11}
\begin{array}{l}
\begin{cases}
y_1(t)=x_1(t)\\
y_2(t)=x_2(t).\\
\end{cases}
\end{array}
\end{align}
Its algebraic form is
\begin{align}\label{5.12}
\begin{array}{l}
y(t)=Hx(t),
\end{array}
\end{align}
where
$$
\begin{array}{ccl}
H&=&\d_4[1,1,2,2,3,3,4,4].
\end{array}
$$
It is easy to figure out that
$$
\begin{array}{cl}
\Theta=&\left\{\{\d_8^1,\d_8^2\}, \{\d_8^3,\d_8^4\},\{\d_8^5,\d_8^6\},\{\d_8^7,\d_8^8\}\right\}\\
~\sim&\left\{\d_{64}^{2}, \d_{64}^{20}, \d_{64}^{38},\d_{64}^{56}\right\}\\
~:=&\{\theta_1,\theta_2, \theta_3,\theta_4\};
\end{array}
$$
and
$$
\begin{array}{cl}
\Xi=&\left\{\{\d_8^1,\d_8^3\},\{\d_8^1,\d_8^4\},\{\d_8^1,\d_8^5\}, \{\d_8^1,\d_8^6\},\{\d_8^1,\d_8^7\},\right.\\
~&~\left.\{\d_8^1,\d_8^8\},\{\d_8^2,\d_8^3\},\{\d_8^2,\d_8^4\},\{\d_8^2,\d_8^5\}, \{\d_8^2,\d_8^6\},\right.\\
~&~\left.\{\d_8^2,\d_8^7\},\{\d_8^2,\d_8^8\},\{\d_8^3,\d_8^5\}, \{\d_8^3,\d_8^6\},\{\d_8^3,\d_8^7\},\right.\\
~&~\left.\{\d_8^3,\d_8^8\},\{\d_8^4,\d_8^5\},\{\d_8^4,\d_8^6\}, \{\d_8^4,\d_8^7\},\{\d_8^4,\d_8^8\},\right.\\
~&~\left.\{\d_8^5,\d_8^7\},\{\d_8^5,\d_8^8\},\{\d_8^6,\d_8^7\},\{\d_8^6,\d_8^8\}\right\}\\
~~\sim&\left\{\d_{64}^{3}, \d_{64}^{4},\d_{64}^{5},\d_{64}^{6},\d_{64}^{7},\d_{64}^{8}, \d_{64}^{11}, \d_{64}^{12}, \d_{64}^{13},\d_{64}^{14},\d_{64}^{15},\right.\\
~&~\left.\d_{64}^{16},\d_{64}^{21},\d_{64}^{22},\d_{64}^{23},\d_{64}^{24},\d_{64}^{29},\d_{64}^{30},\d_{64}^{31},\d_{64}^{32},\d_{64}^{39},\d_{64}^{40},\right.\\
~&~\left.\d_{64}^{47},\d_{64}^{48}\right\}.
\end{array}
$$
Using the initial set $P^0=\{\theta\in\Theta\}=\{\theta_1,\theta_2,\theta_3,\theta_4\}$ and the destination set $P^d=\Xi$, we have
$$
J_d=\sum_{\d_{64}^i\in \Xi}\d_{64}^i;
$$
and
$$
J_0= \d_{64}[2,  20,  38,  56 ].
$$
It follows that
$$
{\cal C}_S=J_d^T{\cal C}J_0=\begin{bmatrix}0&1&0&1\end{bmatrix}.
$$
According to Theorem \ref{t5.2}, system (\ref{5.7}) with outputs (\ref{5.11}) is not observable.
\end{enumerate}
\end{exa}

\section{Conclusion}
In this paper the set controllability of BCN is proposed, and necessary and sufficient condition is obtained. As an application, the output controllability is converted into a set controllability problem and is solved easily. Then an extended system is constructed for a given BCN. It has been proved that the  observability of the given BCN is equivalent to the set controllability of the extended system. Then the observability of a BCN is verified via the set controllability of the extended system by providing a concise and easily verifiable necessary and sufficient condition. A numerical example has been presented to demonstrate the theoretical result. The method reveals a relationship between controllability and observability of BCN.

\appendix[Semi-tensor Product of Matrices]

Semi-tensor product of matrices was proposed by us. It is convenient in dealing with logical functions. We refer to \cite{che11, che12} and the references therein for details. In the follows we give a very brief survey.

\begin{dfn}%
Let $A\in {\cal M}_{m\times n}$ and $B\in {\cal M}_{p\times q}$.
Denote by $
t:=\lcm(n,p)$ the least common multiple of $n$ and $p$.
Then we define the semi-tensor product (STP) of $A$ and $B$ as
\begin{align}\label{a.1}
A\ltimes B:=\left(A\otimes I_{t/n}\right)\left(B\otimes I_{t/p}\right)\in {\cal M}_{(mt/n) \times (qt/p)}.
\end{align}
\end{dfn}

\begin{rem}%
\begin{itemize}
\item When $n=p$, $A\ltimes B=AB$. So the STP is a generalization of conventional matrix product.

\item When $n=rp$, denote it by $A\succ_r B$;

when $rn=p$, denote it by $A\prec_r B$.

These two cases are called the multi-dimensional case, which is particularly important in applications.

\item STP keeps almost all the major properties of the conventional matrix product unchanged.

\end{itemize}
\end{rem}

%
%
%

We cite some basic properties which are used in this note.

\begin{prp}\label{pa.4}
\begin{enumerate}
\item (Associative Low)
\begin{align}\label{a.2}
A\ltimes (B\ltimes C)=(A\ltimes B)\ltimes C.
\end{align}
\item (Distributive Low)
\begin{align}\label{a.3}
\begin{array}{l}
(A + B)\ltimes C=A\ltimes C + B \ltimes C.\\
A\ltimes (B + C)=A\ltimes B + A\ltimes C.
\end{array}
\end{align}
\item
\begin{align}\label{a.4}
(A\ltimes B)^T=B^T\ltimes A^T.
\end{align}
\item Assume $A$ and $B$ are invertible, then
\begin{align}\label{a.5}
(A\ltimes B)^{-1}=B^{-1}\ltimes A^{-1}.
\end{align}
\end{enumerate}
\end{prp}
\begin{prp}\label{pa.5}
Let $X\in \R^t$ be a column vector. Then for a matrix $M$
\begin{align}\label{a.6}
X\ltimes M=\left(I_t\otimes M\right) \ltimes X.
\end{align}
\end{prp}

Finally, we consider how to express a Boolean function into an algebraic form.
\begin{thm}\label{ta.8}
Let $f:{\cal D}^n\ra {\cal D}$ be a Boolean function expressed as
\begin{align}\label{a.9}
y=f(x_1,\cdots,x_n).
\end{align}

Identifying
\begin{align}\label{a.10}
1\sim \d_2^1, \quad 0\sim \d_2^2.
\end{align}
Then there exists a unique logical matrix $M_f\in {\cal L}_{2\times 2^n}$, called the structure matrix of $f$, such that under vector form,
by using (\ref{a.10}), (\ref{a.9}) can be expressed as
\begin{align}\label{a.11}
y=M_f\ltimes_{i=1}^nx_i,
\end{align}
which is called the algebraic form of (\ref{a.9}).
\end{thm}

%
%
%


\begin{thebibliography}{00}
\bibitem{kau69} S.A. Kauffman, Metabolic stability and epigenesis in randomly constructed genetic nets, {\it J. Theoretical Biology}, Vol. 22, No. 3, 437-467, 1969.

\bibitem{aku07} T. Akutsu, M. Hayashida, W.K. Ching, M.K. Ng, Control of boolean networks: hardness results and algorithms for tree structured networks, {\it J. Theor. Biology}, Vol. 244, No. 4, 670-679, 2007.
%
\bibitem{alb00} R. Albert, A.L. Barab\'{a}si, Dynamics of complex systems: Scaling laws for the period of boolean networks, {\it Physical Review Lett.}, Vol. 84, No. 24, 5660-5663, 2000.
%

\bibitem{fau06} A. Faur\'{e}, A. Naldi, C. Chaouiya, D. Thieffry, Dynamical analysis of a generic boolean model for the mammalian cell cycle, {\it Bioinformatics}, Vol. 22, No. 14, 124-131, 2006.

\bibitem{che11} D. Cheng, H. Qi, Z. Li, {\it Analysis and Control of Boolean Networks - A Semi-tensor Product
Approach}, Springer, London, 2011.

\bibitem{che12} D. Cheng, H. Qi, Y. Zhao, {\it An Introduction to Semi-tensor Product of Matrices and Its Applications}, World Scientific, Singapore, 2012.

\bibitem{lujpr} J. Lu, H. Li, Y. Liu, F. Li, A survey on semi-tensor product method with its applications in logical networks and other finite-valued systems, {\it IET Contr. Theory Appl.}, Vol. 11, No. 13, 2040-2047, 2017.

\bibitem{che09} D. Cheng, H. Qi, Controllability and observability of Boolean control networks, {\it Automatica}, Vol. 45, No. 7, 1659-1667, 2009.
%
\bibitem{las12} D. Laschov, M. Margaliot, Controllability of Boolean control networks via the perron-forbenius theory, {\it Automatica}, Vol. 48, No. 6, 1218-1223, 2012.

\bibitem{che11b} D. Cheng, Disturbance decoupling of Boolean control networks, {\it IEEE Trans. Aut. Contr.}, Vol. 56, No. 1, 2-10, 2011.
%

\bibitem{yan13} M. Yang, R. Li, T. Chu, Controller design for disturbance decoupling of Boolean control networks, {\it Automatica}, Vol. 49, No. 1, 273-277, 2013.
%
\bibitem{for14} E. Fornasini, M.E. Valcher, Optimal control of boolean control networks, {\it IEEE Trans. Aut. Contr.}, Vol. 59, No. 5, 1258-1270, 2014.
%
%

%
\bibitem{las11} D. Laschov, M. Margaliot, A maximum pronciple for single-input boolean control networks, {\it IEEE Trans. Aut. Contr.}, Vol. 56, No. 4, 913-917, 2011.
%

\bibitem{che11c} D. Cheng, H. Qi, Z. Li, J. Liu, Stability and stabilization of Boolean networks, {\it Int. J. Rob. Nonlin.  Contr.}, Vol. 21, No. 2, 134-156, 2011.
%
\bibitem{liu16} Y. Liu, J. Cao, L. Sun, J. Lu, Sampled-data state feedback stabilization of Boolean control networks, {\it Neural Computation}, Vol. 28, No. 4, 778-799, 2016.
%

\bibitem{zha10} Y. Zhao, H. Qi, D. Cheng, Input-state incidence matrix of Boolean control networks and its applications, {\it Sys. Contr. Lett.}, Vol. 59, No. 12, 767-774, 2010.
%


%

%
\bibitem{for13} E. Fornasini, M. Valcher, Observability, reconstructibility and state obververs of Boolean control networks, {\it IEEE Trans. Aut. Contr.}, Vol. 58, 1390-1401, 2013.
%


%
\bibitem{las13} D. Laschov, M. Margaliot, G. Even, Observability of Boolean networks: A graph-theoretic approach, {\it Automatica}, Vol. 49, 2351-2362, 2013.


\bibitem{zha16} K. Zhang, L. Zhang, Observability of Boolean control networks: A unified approach based on finite automata, {\it IEEE Trans. Aut. Contr.}, Vol. 61, No. 9, 2733-2738, 2016.

\bibitem{che16} D. Cheng, H. Qi, T. Liu, Y. Wang, A note on observability of Boolean control networks, {\it Sys. Contr. Lett}, Vol. 87, 76-82, 2016.

\bibitem{mch09} M. Chaves, Methods for qualitative analysis of genetic networks, {\it Proc. Eur. Control Conf.}, 671-676, 2009.

\bibitem{oga97} K. Ogata, {\it Modern Control Engineering} (3rd ed.), Prentice-Hall, NJ, 1997.
%
\bibitem{liu14} Y. Liu, J. Lu, B. Wu, Some necessary and sufficient conditions for the output controllability of temporal Boolean control networks, {\it ESAIM: Contr. Opt. Calcul. Variat.}, Vol. 20, 158-173, 2014.

\bibitem{avc11} A. Veliz-Cuba, B. Stigler, Boolean models can explain bistability in the lac operon, {\it J. Comput. Biol.}, Vol. 18, No. 6, 783-794, 2011.
%

\end{thebibliography}
\end{document}